\documentclass[a4paper,10pt]{article}

\usepackage{latexsym} \usepackage{mathrsfs}

\usepackage{amsmath} \usepackage{amssymb} \usepackage{amsfonts}
\usepackage{amsthm} \usepackage{amsxtra} 

\usepackage{hyperref}

\usepackage{enumerate}

\newcommand{\floor}[1]{\left\lfloor{#1}\right\rfloor}
\newcommand{\ceil}[1]{\left\lceil{#1}\right\rceil}
\newcommand{\abs}[1]{\left\lvert{#1}\right\rvert}
\newcommand{\norm}[1]{\left\|{#1}\right\|}

\newcommand{\bb}{\mathbb} 
\newcommand{\mc}{\mathcal} 

\newcommand{\R}{\mathbb{R}}\newcommand{\N}{\mathbb{N}}
\newcommand{\Z}{\mathbb{Z}}\newcommand{\Q}{\mathbb{Q}}
\newcommand{\T}{\mathbb{T}}\newcommand{\C}{\mathbb{C}}

 \newcommand{\AK}{\mathcal{AK}^\infty}
\newcommand{\DUE}{\mathrm{DUE}}

 \newcommand{\ie}{i.e.\ }
\newcommand{\eg}{e.g.\ }

\newtheorem*{mainThmA}{Theorem A} 

\newtheorem{theorem}{Theorem}[section]
\newtheorem{proposition}[theorem]{Proposition}
\newtheorem{lemma}[theorem]{Lemma}

\theoremstyle{definition}

\theoremstyle{remark} 

\newcommand{\Diff}[2]{\mathrm{Diff}_{#1}^{#2}}

\newcommand{\Pol}{\mathfrak{Pol}}
\newcommand{\Hsw}[1]{\mathrm{SW}^{#1}} \newcommand{\cc}{\mathrm{cc}}

\newcommand{\Leb}{\mathrm{Leb}} \newcommand{\dd}{\:\mathrm{d}}
\newcommand{\Dis}[1]{\mathcal{D}^\prime_{#1}}

\DeclareMathOperator{\M}{\mathfrak{M}}

 \DeclareMathOperator{\cl}{cl}

\newcommand{\Bs}{\mathcal{S}} 
\newcommand{\n}{\mathfrak{n}}

\begin{document}

\title{On manifolds supporting distributionally uniquely ergodic
  diffeomorphisms}

\author{Artur Avila\footnote{Institut de Math\'ematiques de Jussieu,
    Paris, France; artur@math.jussieu.fr}, Bassam
  Fayad\footnote{Institut de Math\'ematiques de Jussieu, Paris,
    France; bassam@math.jussieu.fr} and Alejandro
  Kocsard\footnote{IME, Univesidade Federal Fluminense, Niter\'oi,
    Brazil; akocsard@id.uff.br}}

\date{\today}

\maketitle

\begin{abstract}
  A smooth diffeomorphism is said to be \emph{distributionally
    uniquely ergodic} (\emph{DUE} for short) when it is uniquely
  ergodic and its unique invariant probability measure is the only
  invariant distribution (up to multiplication by a constant). Ergodic
  translations on tori are classical examples of DUE diffeomorphisms.
  In this article we construct DUE diffeomorphisms supported on closed
  manifolds different from tori, providing some counterexamples to a
  conjecture proposed by Forni in \cite{ForniKatokConj}.
\end{abstract}

\section{Introduction}
\label{sec:intro}

When we study the dynamics of a homeomorphism $f\colon M\to M$ (for
the time being we can suppose $M$ is a just compact metric space), we
can consider the induced linear automorphism $f^\star$ on $C^0(M,\C)$
given by
\begin{displaymath}
  f^\star\psi:=\psi\circ f,\quad\forall\psi\in C^0(M,\C).
\end{displaymath}

If we endow $C^0(M,\C)$ with the $C^0$-uniform topology, $f^\star$
turns to be a continuous linear operator and hence, its adjoint
$f_\star$ acts on the topological dual space $(C^0(M,\C))^\prime$
which coincides, by Riesz representation theorem, with $\M(M)$, the
space of complex finite measures on $M$.

At certain extent we can say that Ergodic Theory consists in
understanding the relation between the \emph{``non-linear''} dynamics
of $f$ and the linear one of $f_\star\colon\M(M)\to\M(M)$. The fixed
points of $f_\star$, the so called $f$-invariant measures, play a key
role in this theory.

When $M$ is a closed smooth manifold and $f\colon M\to M$ is a
$C^r$-diffeomorphism, every linear subspace $C^k(M,\C)\subset
C^0(M,\C)$ (where $0\leq k\leq r\leq\infty$) is
$f^\star$-invariant. Moreover, when $C^k(M,\C)$ is equipped with the
$C^k$-uniform topology, $f^\star\colon C^k(M,\C)\to C^k(M,\C)$ turns
to be a continuous isomorphism and hence, its adjoint $f_\star$ acts
on $\Dis{k}(M)$, \ie the space of distributions up to order $k$. Of
course, the fixed points of $f_\star$ are called \emph{invariant
  distributions.}

As usual, we say $f$ is uniquely ergodic when it exhibits a unique
$f$-invariant probability measure. On the other hand, when $f$ is
$C^\infty$ and there is only one (up to multiplication by constant)
$f$-invariant distribution, we shall say that $f$ is
\emph{distributionally uniquely ergodic}, \emph{DUE} for short.
Ergodic translations on tori are the archetypical examples of DUE
diffeomorphisms. Recently, the first and third authors showed in
\cite{AviKocCohomoEqInvDistCirc} that every smooth circle
diffeomorphism with irrational rotation number is also DUE.

In 2008, Forni conjectured in \cite{ForniKatokConj} that tori are the
only closed manifolds supporting DUE diffeomorphisms.

In this paper we construct some new DUE systems, providing some
counterexamples to Forni's conjecture.  In fact, our main purpose
consists in showing the following 
\begin{mainThmA}
  Let $P$ be either 
  \begin{enumerate}[(a)]
  \item a compact nilmanifold, \ie $P=N/\Gamma$ with $N$ a nilpotent
    connected and simply connected Lie group and $\Gamma< N$ a uniform
    lattice; 
  \item or a homogeneous space of compact type, \ie $P=G/H$ where $G$
    is compact Lie group and $H<G$ a closed subgroup.
  \end{enumerate}
  Then, there exist DUE diffeomorphisms on $M:=\T\times P$.
\end{mainThmA}

It is interesting to remark that so far the most powerful techniques
to study invariant distributions (for dynamical systems which are not
hyperbolic) come from harmonic analysis. However, in general it is
very hard to apply these techniques to dynamical systems which do not
exhibit certain \emph{``homogeneity''} (\eg they preserve a smooth
Riemannian structure, or are induced by translations on homogeneous
spaces).

On the other hand, it is well-known that any DUE diffeomorphism
preserving a Riemannian structure is topologically conjugate to an
ergodic torus translation, and after some works of Flaminio and Forni
\cite{FlamForniInvDistrHorocycle,FlamForniNilflows} it was expected
that there were no DUE homogeneous systems supported on homogeneous
spaces different from tori. In fact, we recently learned that
Flaminio, Forni and F. Rodr\'\i{}guez-Hertz~\cite{fede} have shown
indeed the validity of Forni's conjecture for homogeneous systems. So
the main difficulty to prove our result consists in overcoming this
apparent obstruction to apply harmonic analysis tools.

\section*{Acknowledgments}
Since most part of this work was done while A. Kocsard was visiting
the Institut de Math\'ematique de Jussieu, in Paris, he would like to
thank P. Le Calvez for the hospitality. He is also grateful to
L. Flaminio for some insightful discussions.  A.K. was partially
supported by CAPES and CNPq (Brazil).

\section{Preliminaries and Notations}
\label{sec:prelim-notat}

\subsection{Manifolds, functional spaces and topology}
\label{sec:manif-funct-spac}

All along this paper, $M$ will denote a compact orientable smooth
manifold without boundary. Given any $r\in\bb{N}_0\cup\{\infty\}$, we
write $\Diff{}r(M)$ for the group of $C^r$-diffeomorphisms. The
subgroup of $C^r$-diffeomorphisms which are isotopic to the identity
shall be denoted by $\Diff0r(M)$.

If $N$ denotes any other smooth manifold, we write $C^r(M,N)$ for the
space of $C^r$-maps from $M$ to $N$. For the sake of simplicity, we
shall just write $C^r(M)$ instead of $C^r(M,\C)$.

Let us recall that, when $r$ is finite, the \emph{uniform
  $C^r$-topology} turns $C^r(M)$ into a Banach space and $C^r(M,N)$
turns to be a Banach manifold.

The space $C^\infty(M)$ will be endowed with its usual Fr\'echet
topology which can be defined as the projective limit of the family of
Banach spaces $(C^r(M))_{r\in\bb{N}}$. In this case, $C^\infty(M,N)$
is endowed with a Fr\'echet manifold structure.

Of course, for any $r\in\N_0\cup\{\infty\}$, we assume $\Diff{}r(M)$
equipped with the $C^r$-uniform topology inherited from its inclusion
in $C^r(M,M)$.

Finally, if $X$ is an arbitrary topological space and $z\in X$, we
shall write $\cc(X,z)$ to denote the connected component of $X$
containing $z$.

\subsection{Some arithmetical notations}
\label{sec:arithm}

Given any natural number $q\in\N$, we write $q\N:=\{qn: n\in\N\}$.

Whenever we write a single rational number in the form $p/q$ we always
assume the integers $p$ and $q$ are coprime, \ie $1$ is the greatest
common divisor of $p$ and $q$. On the other hand, writing a vector
with rational coordinates $(p_1/q,\ldots,p_n/q)\in\Q^n$, we shall
simply assume $\gcd(p_1,\ldots,p_n,q)=1$.

Given any $x\in\R$, we write $\floor{x}$ to denote the largest integer
not greater than $x$. Analogously, $\ceil{x}$ denotes the smallest
integer greater or equal than $x$.

For any $\boldsymbol{\alpha}\in\bb{R}^d$ we write
\begin{displaymath}
  \norm{\boldsymbol{\alpha}}_{\T^d}:=\mathrm{dist}
  (\boldsymbol{\alpha},\bb{Z}^d).
\end{displaymath}
Notice that, since $\norm{\boldsymbol{\alpha}+\mathbf{n}}_{\T^d} =
\norm{\boldsymbol{\alpha}}_{\T^d}$ for every $\mathbf{n}\in\bb{Z}^d$,
we can naturally consider $\|\cdot\|_{\T^d}$ as defined on $\bb{T}^d$,
too.

We say $\boldsymbol{\alpha}=(\alpha_1,\ldots,\alpha_d)\in\bb{R}^d$ is
\emph{irrational} when for every $(n_1,\ldots,n_d)\in\bb{Z}^{d}$
\begin{equation}
  \label{eq:dioph-vector-def}
  \bigg\|\sum_{i=1}^{d} n_i\alpha_i\bigg\|_{\T^d} =0 \implies n_i=0,\
  \text{for } i=1,\ldots, d.
\end{equation}

An irrational vector $\boldsymbol{\alpha}$ is said to be
\emph{Diophantine} if there exist constants $C,\tau>0$ satisfying
\begin{displaymath}
  \bigg\|\sum_{i=1}^d\alpha_iq_i\bigg\|_{\T^d}
  \geq\frac{C}{\max_i|q_i|^\tau}, 
\end{displaymath}
for every $(q_1,\ldots,q_d)\in\bb{Z}^d\setminus\{0\}$. On the other
hand, an irrational element of $\bb{R}^d$ which is not Diophantine is
called \emph{Liouville}.

\subsection{Lie groups}
\label{sec:Lie-groups}

\subsubsection{Generalities}
\label{sec:Lie-gr-generalities}

In this work we shall only deal with real connected Lie groups. As
usual, if $G$ denotes an arbitrary Lie group, its identity element is
denoted by $1_G$\footnote{Except when $G$ is abelian. In that case, we
  just write $0$ to denote its identity element.}, its Lie algebra by
$\mathfrak{g}$ and we write $\exp\colon\mathfrak{g}\to G$ for the
exponential map.

A smooth manifold is called a \emph{homogeneous space} when it can be
written as $G/H$, where $G$ denotes a (real, connected) Lie group and
$H<G$ a closed subgroup. We say $H$ is \emph{cocompact} when $G/H$ is
compact, and we say that $G/H$ is of \emph{compact type} when $G$ is
compact itself.

Clearly, the group $G$ acts naturally (on the left) on $G/H$ and it is
well known that in such a case there exists at most one $G$-invariant
Borel probability measure on $G/H$. When such a measure does exist, we
will call it the \emph{Haar measure of} $G/H$. A discrete
cocompact subgroup will be called a \emph{uniform lattice.} Let us
recall that the existence of the Haar measure on $G/H$ is guaranteed
whenever either $G$ is compact, or $H$ is a uniform lattice.

Making some abuse of notation, we will use the brackets
$[\cdot,\cdot]$ to denote the Lie brackets on $\mathfrak{g}$, as well
as the commutator operator in $G$, \ie $[g,h]:=ghg^{-1}h^{-1}$, for
any $g,h\in G$.  

More generally, if $A,B\subset G$ we define $[A,B]:=\left\langle[a,b]
  : a\in A,\ b\in B\right\rangle$, \ie the (abstract) subgroup of $G$
generated by commutators of the set subsets $A$ and $B$,
respectively. And analogously, if
$\mathfrak{h},\mathfrak{k}\subset\mathfrak{g}$, we define
$[\mathfrak{h},\mathfrak{k}]:=\mathrm{span}_\R\{[v,w] :
v\in\mathfrak{h},\ w\in\mathfrak{k}\}$.

The \emph{centers} of $G$ and $\mathfrak{g}$ are defined by
$Z(G):=\{g\in G : [g,h]=1_G,\ \forall h\in G\}$ and
$Z(\mathfrak{g}):=\{v\in\mathfrak{g} : [v,w]=0,\ \forall
w\in\mathfrak{g}\}$, respectively.

\subsubsection{Tori}
\label{sec:tori}

The $d$-dimensional torus will be denoted by $\T^d$ and will be
identified with $\R^d/\Z^d$. The canonical quotient projection will be
denoted by $\pi\colon\R^d\to\T^d$. For simplicity, we shall simply
write $\T$ for the $1$-torus, \ie the circle.

The symbol $\Leb_d$ will be used to denote the Lebesgue measure on
$\R^d$, as well as the Haar measure on $\T^d$. Once again, for the
sake of simplicity, we just write $\Leb$, and also $\dd x$, instead of
$\Leb_1$.

For each $\alpha\in\T^d$, let $R_\alpha\colon\T^d\to\T^d$ be the rigid
translation $R_\alpha : x\mapsto x+\alpha$.


\subsubsection{Homogeneous skew-products}
\label{sec:homog-skew-prod}

Given an arbitrary Lie group $G$ and any closed subgroup $H<G$, for
any $\alpha\in\T$ and any $\gamma\in C^r(\T,G)$, we define the
\emph{homogeneous skew-product}
$H_{\alpha,\gamma}\in\Diff{}{r}(\T\times G/H)$ by
\begin{displaymath}
  H_{\alpha,\gamma} : (t,gH) \mapsto (t+\alpha,\gamma(t)gH),
  \quad\forall (t,gH)\in \T\times G/H. 
\end{displaymath}

The space of $C^r$ homogeneous skew-products on $\T\times G/H$ shall
be denoted by $\Hsw{r}(\T\times G/H)$. If $G/H$ admits a Haar measure,
and we denote it by $\nu$, then we clearly have $\Hsw{r}(\T\times
G/H)\subset \Diff{\mu}{r}(\T\times G/H)$, where
$\mu:=\Leb\otimes\nu$.

\subsubsection{Nilmanifolds and Mal'cev theory}
\label{sec:malcev-theory}

Given an arbitrary Lie group $G$, the \emph{central descending series
  of } $G$ can be recursively defined by $G_{0}:=G$ and
\begin{displaymath}
  G_{n}:=[G,G_{n-1}],\quad \forall n\geq 1.
\end{displaymath}
We say $G$ is \emph{nilpotent} when $G_{k}=\{1_G\}$, for certain
$k\in\N$. The \emph{degree of nilpotency} of $G$ is defined as the
maximal natural number $n$ such that $G_n\neq\{1_G\}$.

From now on, $N$ shall denote a (connected) simply connected nilpotent
Lie group admitting a uniform lattice $\Gamma<N$. As usual, the
(compact) homogeneous space $N/\Gamma$ is called a (compact)
\emph{nilmanifold.}

It is important to recall that in such a case the exponential map
$\exp\colon\mathfrak{n}\to N$ is a real-analytic diffeomorphism (see
Theorem 1.2.1 in \cite{CorwinGreenleaf}). Hence, in this case $N$ as
well as $\n$ can be identified with the universal cover of $N/\Gamma$.

After Mal'cev \cite{MalcevHomogenSpaces}, a basis
$\{v_1,v_2,\ldots,v_d\}$ of the Lie algebra $\n$ is called a
\emph{Mal'cev basis} whenever
$\n_{(i)}:=\mathrm{span}_\R\{v_1,\ldots,v_i\}$ is an ideal in $\n$, for
each $i\in\{1,\ldots,d\}$. Moreover, such a basis is said to be a
\emph{strongly based on} $\Gamma$ when
\begin{equation}
  \label{eq:malcev-basis-Gamma-prop}
  \Gamma=\exp(\Z v_1)\exp(\Z v_2)\ldots\exp(\Z v_d).
\end{equation}

In \cite{MalcevHomogenSpaces} (see also \cite{CorwinGreenleaf}) it is
proved that there always exists a Mal'cev basis strongly based on
$\Gamma$ when $N$ and $\Gamma$ are as above.

Since each $\n_{(i)}$ is an ideal in $\n$,
$N_{(i)}:=\exp(\n_{(i)})\subset N$ turns to be a (closed) normal
subgroup of $N$, and the quotient $N^{(i)}:=N/N_{(i)}$ a nilpotent
connected and simply connected Lie group, itself.

On the other hand, as a consequence of
\eqref{eq:malcev-basis-Gamma-prop} we have
\begin{displaymath}
  \Gamma_{(i)}:=\exp\Big(\Z v_1\oplus\Z v_2\oplus\cdots\oplus\Z
  v_i\}\Big) \subset N_{(i)} 
\end{displaymath}
is a discrete subgroup of $\Gamma$. Hence,
$\Gamma^{(i)}:=\Gamma/\Gamma_{(i)}$ can be naturally identified with a
uniform lattice of $N^{(i)}$.

\subsection{Distributions and distributional unique ergodicity}
\label{sec:distributions}

Given any $k\in\bb{N}_0$, the space of \emph{distribution on $M$ up to
  order $k$} is defined as the topological dual space of $C^k(M)$ and
will be denoted by $\Dis{k}(M)$. When $k=0$, by Riesz representation
theorem $\Dis{0}(M)$ can be identified with the space of finite
complex measures on $M$ and so, it will also be denoted by $\M(M)$.

On the other hand, as it is usually done, the topological dual space
of $C^\infty(M)$ will be simply denoted by $\Dis{}(M)$ and its
elements are just called distributions.

Since all the inclusions $C^{k+1}(M)\hookrightarrow C^k(M)$ and
$C^\infty(M)\hookrightarrow C^k(M)$ are continuous, making some abuse
of notation we can consider the following chain of inclusions (modulo
restrictions):
\begin{displaymath}
  \M(M)=\Dis0(M)\subset\Dis1(M)\subset\Dis2(M)\subset\ldots\subset
  \Dis{}(M).
\end{displaymath}
Moreover, since we are assuming $M$ is compact, it is well-known that
\begin{displaymath}
  \Dis{}(M)=\bigcup_{k\geq 0}\Dis{k}(M).
\end{displaymath}

Now, as it was already mentioned in \S\ref{sec:intro}, any
$f\in\Diff{}k(M)$ acts linearly on $C^k(M)$ by pull-back, and the
adjoint of this action is the linear operator
$f_\star\colon\Dis{k}(M)\to\Dis{k}(M)$ given by
\begin{displaymath}
  \langle f_\star T,\psi\rangle:=\langle T,f^\star\psi\rangle= \langle 
  T,\psi\circ f\rangle,\quad 
  \forall\:T\in\Dis{k}(M),\ \forall\:\psi\in C^k(M).
\end{displaymath}

The space of \emph{$f$-invariant distributions up to order $k$}
is defined by 
\begin{displaymath}
  \Dis{k}(f):=\{T\in\Dis{k}(M) : f_\star T=T\}.
\end{displaymath}
Of course, when $f$ is $C^\infty$ we write $\Dis{}(f):=\bigcup_{k\geq
  0}\Dis{k}(f)$.

Given any measure $\mu\in\M(M)$ and $r\geq 0$, we define
\begin{displaymath}
  \Diff\mu{r}(M):=\{f\in\Diff{}r(M) : f_\star\mu=\mu\},
\end{displaymath}
and 
\begin{displaymath}
  C^r_\mu(M):=\left\{\phi\in C^r(M) :
    \int_M\phi\dd\mu=0\right\}.
\end{displaymath}

As usual, we say that $f$ is uniquely ergodic when $\M(f):=\Dis0(f)$
is one-dimensional. We will say that $f$ is \emph{distributionally
  uniquely ergodic} (or just \emph{DUE} for short) when $f$ is
$C^\infty$ and $\Dis{}(f)$ has dimension one.

\subsubsection{Coboundaries and distributions}
\label{sec:coub-distrib}

Given any $f\in\Diff{}r(M)$, any $\psi\colon M\to\C$ and $n\in\Z$, the
\emph{Birkhoff sum} is defined by
\begin{displaymath}
  \Bs^n\psi=\Bs^n_f\psi:=
  \begin{cases}
    \sum_{i=0}^{n-1}\psi\circ f^i & \text{if } n\geq 1; \\
    0 & \text{if } n=0; \\
    -\sum_{i=1}^{-n}\psi\circ f^{-i} & \text{if } n<0.
  \end{cases}
\end{displaymath}

We say that $\psi$ is a $C^\ell$-\emph{coboundary} for $f$ (with
$0\leq\ell\leq r$) whenever there exists $u\in C^\ell(M)$ solving the
following cohomological equation:
\begin{displaymath}
  u\circ f -u =\psi.
\end{displaymath}
Observe that in this case it holds $\Bs^n\psi=uf^n-u$, for any
$n\in\Z$.


The space of $C^\ell$-coboundaries will be denoted by
$B(f,C^\ell(M))$. Following Katok~\cite{KatokRobinson}, we say $f$ is
\emph{cohomologically $C^\ell$-stable} whenever $B(f,C^\ell(M))$ is a
closed in $C^\ell(M)$.

Finally, notice that as a straightforward consequence of Hahn-Banach
theorem we get
\begin{proposition}
  \label{pro:inv-dist-vs-cobound}
  Given any $f\in\Diff{}{k}(M)$, with $k\in\bb{N}_0\cup\{\infty\}$, it
  holds
  \begin{displaymath}
    \cl_{k}(B(f,C^k(M)))=\bigcap_{T\in\Dis{k}(f)}\ker T,
  \end{displaymath}
  where $\cl_k(\cdot)$ denotes the closure in $C^k(M)$.
\end{proposition}

\subsubsection{Unique ergodicity vs. DUE}
\label{sec:ue-vs-due}

There are many well-known examples of uniquely ergodic systems which
are not DUE. Maybe, the simplest one is given by the \emph{parabolic}
map 
\begin{displaymath}
  \T^2\ni (x,y)\mapsto (x+\alpha,y+x),
\end{displaymath}
with $\alpha\in\T\setminus(\Q/\Z)$ (see \cite{KatokRobinson} for
details). Horocycle flows on constant negatively curved closed
surfaces and minimal homogeneous flows on closed nilmanifolds
different from tori are more elaborated examples
\cite{FlamForniInvDistrHorocycle,FlamForniNilflows}.

On the other hand, a classical result due to Kronecker affirms that a
translation $R_\alpha\colon\T^d\to\T^d$ is uniquely ergodic ($\Leb_d$
is the only $R_\alpha$-invariant probability measure) if and only
$\alpha=(\alpha_1,\ldots,\alpha_d)$ is irrational.

Moreover, we have the following result which belongs to the
\emph{folklore}:
\begin{proposition}
  \label{pro:one-dim-inv-dist}
  $R_\alpha$ is DUE, and it is cohomologically $C^\infty$-stable if
  and only if $\alpha$ is a Diophantine vector.
\end{proposition}

\begin{proof}
  See Proposition 2.3 in \cite{AviKocCohomoEqInvDistCirc} for a proof.
\end{proof}

Recently, the first and third authors extended this result in
\cite{AviKocCohomoEqInvDistCirc} showing that any minimal circle
diffeomorphism is also DUE (see \cite{NavasTriestinoInvDist} for a
much simpler proof of non-existence of invariant distributions up to
order $1$).

As it was already mentioned in \S\ref{sec:intro}, the main aim of
this paper consists in constructing DUE diffeomorphism which are not
topologically conjugate to ergodic translations on tori.

\section{Proof of Theorem A: general strategy}
\label{sec:strat-proof}

As it was already mentioned in \S\ref{sec:Lie-groups}, in both cases
considered in Theorem A the homogeneous space $P$ admits a Haar
measure that will be denoted by $\nu_P$. Then, the Haar measure of
$M$, which is a homogeneous space itself, is given by
$\mu:=\Leb_1\otimes\nu_P$.

Let us recall that any homogeneous skew-product on $M=\T\times P$ (see
\S\ref{sec:homog-skew-prod}) preserves the measure $\mu$. In other
words, $\Hsw{r}(\T\times P)\subset\Diff{\mu}{r}(M)$.

\subsection{The Anosov-Katok space}
\label{sec:anosov-katok-space}

Let us consider the \emph{horizontal} $\T$-action $T\colon\T\times
M\to M$ given by 
\begin{displaymath}
  T_\alpha(t,p)=T\big(\alpha,(t,p)\big)=(t+\alpha,p),\quad \forall
  (t,p)\in M=\T\times P.
\end{displaymath}

Then we define the \emph{Anosov-Katok space}
\begin{equation}
  \label{eq:anosov-katok-space-smooth}
  \AK(T):=\cl_\infty\left\{H\circ T_\alpha\circ
    H^{-1} : \alpha\in\T,\
    H\in\Hsw{\infty}(\T\times P)\right\}.
\end{equation}
Observe that each $T_\alpha\in\Hsw\infty(\T\times P)$ and hence,
$\AK(T)\subset\Hsw{\infty}(\T\times P)$.

To prove Theorem A we will show 
\begin{theorem} 
  \label{thm:mainThm-AK}
  Generic diffeomorphisms in $\AK(T)$ are DUE. More precisely, the set 
\begin{displaymath}
  \DUE(T):=\left\{f\in\AK(T) : \dim\Dis{}(f)=1\right\}
\end{displaymath}
contains a dense $G_\delta$-subset of $\AK(T)$.
\end{theorem}

Let us now describe the general strategy to prove
Theorem~\ref{thm:mainThm-AK}: 

A family $(V_n)_{n\geq 1}$ will be called a \emph{filtration} of
$C^\infty_\mu(M)$ whenever it satisfies:
\begin{itemize}
\item for every $n\geq 1$, $V_n\subset C^\infty_\mu(M)$ is a closed
  linear subspace,
\item $V_n\subset V_{n+1}$, for every $n\geq 1$;
\item $\bigcup_{n\geq 1} V_n$ is dense in $C^\infty_\mu(M)$. 
\end{itemize}

Sections \ref{sec:proof-nilpotent-case} and
\ref{sec:proof-compact-case} are dedicated to the proof of the
following lemma, in the nilpotent and the compact cases, respectively: 

\begin{lemma}
  \label{lem:mainLem} If $P$ is as in Theorem A, then there exists a
  filtration of $C^\infty_\mu(M)$, called $(V_k)$, satisfying the
  following condition:

  For every $n\in\N$ and every $q_0\in\N$, there exists $\bar q\in\N$
  and a homogeneous skew-product $H_{0,\gamma}\in\Hsw{\infty}(\T\times
  P)$ such that:
  \begin{enumerate}[(i)]
  \item $H_{0,\gamma}\circ T_{1/q_0} = T_{1/q_0}\circ H_{0,\gamma}$;
  \item $V_n\subset B\big(H_{0,\gamma}\circ T_{\bar p/\bar q}\circ
    H_{0,\gamma}^{-1}, C^\infty(M))$, for every $\bar p\in\Z$ coprime
    with $\bar q$.
  \end{enumerate}
\end{lemma}

To prove Lemma~\ref{lem:mainLem}, we shall need the following
elementary
\begin{lemma}
  \label{lem:cobound-periodic-map}
  Let $M$ be an arbitrary manifold, $f\colon M\to M$ a periodic
  $C^r$-map (\ie there exists $q\in\N$ such that $f^q=id_M$) and
  $\phi\in C^k(M)$, with $0\leq k\leq r\leq\infty$. Then, $\phi\in
  B(f,C^k(M))$ iff 
  \begin{equation}
    \label{eq:Bs-vanish-iff-cobound}
    \Bs_f^q\phi(x)=\sum_{j=0}^{q-1}\phi(f^j(x))=0, \quad\forall x\in
    M. 
  \end{equation}
\end{lemma}

\begin{proof}[Proof of Lemma~\ref{lem:cobound-periodic-map}]
  If $\phi\in B(f,C^k(M))$, then there exists $u\in C^k(M)$ such that
  $\phi=u\circ f-u$. Hence, $\Bs_f^q\phi(x)=u(f^q(x))-u(x)=0$, for
  every $x\in M$.

  Reciprocally, let us suppose \eqref{eq:Bs-vanish-iff-cobound}
  holds. Then, using a formula we learned from
  \cite{MoulinPinchonSyst}\footnote{We thank A. Navas for bringing
    this equation to our attention.}, we write
  \begin{displaymath}
    v(x):=-\frac{1}{q}\sum_{j=1}^q\Bs_f^j\phi(x),\quad \forall
    x\in M.
  \end{displaymath}
  It clearly holds $v\in C^k(M)$, and
  \begin{displaymath}
    \begin{split}
      v(f(x))-v(x)&=-\frac{1}{q}\bigg(\sum_{j=1}^q\big(
      \Bs_f^j\phi(f(x))-\Bs_f^j\phi(x)\big)\bigg) \\
      & = -\frac{1}{q}\bigg(\Bs_f^q\phi(f(x))-q\phi(x)\bigg)=\phi(x),
    \end{split}
  \end{displaymath}
  for every $x\in M$. Thus, $\phi\in B(f,C^k(M))$.
\end{proof}

Now, assuming Lemma~\ref{lem:mainLem}, we can prove
Theorem~\ref{thm:mainThm-AK}, and henceforth, Theorem A, too:

\begin{proof}[Proof of Theorem~\ref{thm:mainThm-AK}]
  Let $(\phi_m)_{m\in\N}$ be a dense sequence in $C^\infty_\mu(M)$ and
  define 
  \begin{displaymath}
    A_m:=\left\{f\in\AK(T) : \exists u\in C^\infty(M),\
      \norm{uf-u-\phi_m}_{C^m}<\frac{1}{m}\right\}.  
  \end{displaymath}
  Each set $A_m$ is clearly open in $\AK(T)$, and by
  Proposition~\ref{pro:inv-dist-vs-cobound}, it holds 
  \begin{displaymath}
    \DUE(T) = \bigcap_{m\geq 1} A_m
  \end{displaymath}
  Thus, we have to show each $A_m$ is dense in $\AK(T)$. To do that,
  consider a fixed set $A_m$, any rational number $p_0/q_0$ and an
  arbitrary homogeneous skew-product $H\in\Hsw\infty(\T\times P)$.

  Since $(V_n)_{n\geq 1}$ is a filtration, there exists $n\in\N$ and
  $\phi\in V_n$ such that
  \begin{equation}
    \label{eq:pol-aprox-phi-n}
    \norm{\phi\circ H^{-1}-\phi_m}_{C^m}<\frac{1}{m}.
  \end{equation}

  Now, invoking Lemma~\ref{lem:mainLem}, from $n$ and $q_0$ we obtain
  a natural number $\bar q$ and a homogeneous skew-product
  $H_{0,\gamma}$ satisfying (i) and (ii).  Then, for each $\ell\in\N$,
  let us define
  \begin{align*}
    \hat p_\ell&:=\bar qp_0\ell +1,\\
    \hat q_\ell&:=\bar qq_0\ell,\\
    p_\ell&:=\frac{\hat p_\ell}{\gcd(\hat p_\ell,\hat q_\ell)}, \\
    q_\ell&:=\frac{\hat q_\ell}{\gcd(\hat p_\ell,\hat q_\ell)}.
  \end{align*}
  Notice that, for each $\ell$, $p_\ell$ and $q_\ell$ are coprime,
  $q_\ell$ is multiple of $\bar q$ and
  $\frac{p_\ell}{q_\ell}\to\frac{p_0}{q_0}$, as $\ell\to+\infty$.

  Then observe that, for every $\ell\in\N$ and any $(t,x)\in\T\times
  P$ it holds:
  \begin{equation}
    \label{eq:ph-cobound-T-pell-qell}
    \begin{split}
      \Bs_{H_{0,\gamma}T_{p_\ell/q_\ell}H^{-1}_{0,\gamma}}^{q_\ell}\phi(t,x)
      & = \sum_{j=0}^{q_\ell-1}\phi\bigg(t+\frac{jp_\ell}{q_\ell},
      \gamma\Big(t+\frac{jp_\ell}{q_\ell}\Big)\gamma(t)^{-1} x\bigg)
      \\
      &= \sum_{j=0}^{q_\ell-1}\phi\bigg(t+\frac{j}{q_\ell},
      \gamma\Big(t+\frac{j}{q_\ell}\Big)\gamma(t)^{-1} x\bigg) \\
      &= \sum_{r=0}^{q_\ell/\bar q -1}\sum_{s=0}^{\bar
        q-1}\phi\bigg(t+\frac{s}{\bar q} + \frac{r}{q_\ell},
      \gamma\Big(t+\frac{s}{\bar q} + \frac{r}{q_\ell}\Big)
      \gamma(t)^{-1}x\bigg) \\
      &= 0,
    \end{split}
  \end{equation}
  where las equality is consequence of condition (ii) of
  Lemma~\ref{lem:mainLem} and
  Lemma~\ref{lem:cobound-periodic-map}. Thus, we conclude that
  $\phi\in B(H_{0,\gamma}T_{p_\ell/q_\ell}H^{-1}_{0,\gamma},
  C^\infty(M))$.

  Henceforth,
  \begin{equation}
    \label{eq:phi-H-inv-cobound}
    \phi\circ H^{-1}\in B\left(H H_{0,\gamma} T_{p_\ell/q_\ell}
      H_{0,\gamma}^{-1} H^{-1}, C^\infty(M) \right),
  \end{equation}
  for every $\ell\in\N$.

  On the other hand, $T_{p_\ell/q_\ell}\to T_{p_0/q_0}$ in
  $\Diff{}{\infty}(M)$, as $\ell\to\infty$. Hence, from (i) of
  Lemma~\ref{lem:mainLem}, we get
  \begin{equation}
    \label{eq:convergence-AK-meth}
    HH_{0,\gamma}T_{p_\ell/q_\ell}H_{0,\gamma}^{-1}H^{-1}\xrightarrow{C^\infty}
    HT_{p_0/q_0}H^{-1}, \quad\text{as }\ell\to\infty.
  \end{equation}

  Now, putting together \eqref{eq:pol-aprox-phi-n},
  \eqref{eq:phi-H-inv-cobound} and \eqref{eq:convergence-AK-meth} we
  conclude $HT_{p_0/q_0}H^{-1}\in\cl_\infty(A_m)$, as desired. 
\end{proof}

\subsection{Real-analytic DUE diffeomorphisms}
\label{sec:real-analytic-due-diff}

Before starting with the proof of Lemma~\ref{lem:mainLem}, it is
interesting to remark that using the techniques we applied in
\S\ref{sec:anosov-katok-space} it is possible to prove the existence
of real-analytic DUE diffeomorphisms on $M=\T\times P$.

In fact, for the time being let us suppose $G$ is an arbitrary Lie
group and let us write $G^\C$ for the complexification of $G$. 

Then, for each $\Delta >0$, let us define $C^\omega_\Delta(\T,G)$ as
the set of real-analytic functions $\gamma\colon\T\to G$ that admit a
holomorphic extension from the complex band $A_\Delta:=\{z\in\C :
\abs{\mathrm{Im}\; z}\leq\Delta\}/\Z$ to $G^\C$. Let us consider
$C^\omega_\Delta(\T\,G)$ endowed with the distance dunction $d_\Delta$
given by
\begin{displaymath}
  d_\Delta(\gamma_0,\gamma_1):= \sup_{z\in A_\Delta}
  d_{G^\C}(\gamma_0(z),\gamma_1(z)), \quad\forall\gamma_0,\gamma_1\in
  C^\omega_\Delta(\T,G),
\end{displaymath}
where $d_{G^\C}$ denotes a left invariant distance on $G^\C$.

Then, taking into account that, for any $\Delta>0$,
$C^\omega_\Delta(\T,G)$ is dense in $C^\infty(\T,G)$, repeating the
same argument used in the Proof of Theorem~\ref{thm:mainThm-AK}, we
can easisly show that the set 
\begin{displaymath}
  \DUE_\Delta^\omega(T):=\left\{(\alpha,\gamma)\in
    C^\omega_\Delta(\T,G) : H_{0,\gamma}\circ T_\alpha\circ
    H_{0,\gamma}^{-1}\in\Hsw\omega (\T\times P) \text{ is DUE}
  \right\} 
\end{displaymath}
is generic in $\T\times C^\omega_\Delta(\T,G)$, and in particular,
non-empty.

\section{The nilpotent case}
\label{sec:proof-nilpotent-case}

All along this section, let us assume $P$ is a compact nilmanifold
equal to $N/\Gamma$, where $N$ is a (connected) simply connected
nilpotent Lie group and $\Gamma<N$ is a uniform lattice.

\subsection{The filtration in the nilpotent case}
\label{sec:pseudo-polyn-nilm}

Observe that any complex function on $P$ can be lifted to its
universal covering, which can be identify with $N$ itself, getting a
$\Gamma$-invariant complex function\footnote{Consider the
  $\Gamma$-action on $N$ given by right translations.}. So, we can
naturally identify $C^\infty(P)$ with
\begin{displaymath}
  C_\Gamma^\infty(N):=\left\{\phi\in C^r(N) : \phi(xg)=\phi(x),\ \forall
    x\in N,\ \forall g\in\Gamma\right\}. 
\end{displaymath}
Moreover, since the exponential map $\exp\colon\n\to N$ is a
real-analytic diffeomorphism, we can identify $C^\infty(N)$ with
$C^\infty(\n)$, and henceforth, $C^\infty(P)$ with a closed linear
subspace of $C^\infty(N)=C^\infty(\n)$.

Let $\mc{V}=\{v_1,\ldots,v_d\}$ be a Mal'cev basis of $\n$ strongly
based on $\Gamma$ (see \S\ref{sec:malcev-theory} for details). Fixing
this basis, we identify $C^\infty(\n)$ with $C^\infty(\R^d)$ simply
writing
\begin{displaymath}
  \phi\left(\sum_{i=1}^dx_i v_i\right)=\phi(x_1,x_2,\ldots,x_d),
  \quad\forall (x_1,\ldots,x_d)\in\R^d.
\end{displaymath}
Thus, making some abuse of notation, we shall assume that
\begin{equation}
  \label{eq:funct-space-indent-1}
  C^\infty(P)\subset C^\infty(N)=C^\infty(\n)=C^\infty(\R^d)
\end{equation}

Now let us analyze some of the equivariant conditions a function
$\phi\in C^\infty(\R^d)$ must satisfy to belong to
$C^\infty(P)$. First, since $v_1\in Z(\n)$ (and $\exp(\Z
v_1)\in\Gamma$), we conclude that, if $\phi\in C^\infty(P)\subset
C^\infty(\R^d)$, then it is $\Z$-periodic in its first
variable. Hence, we can consider the \emph{Fourier-like} development
\begin{equation}
  \label{eq:hat-phi-1-def}
  \phi(x_1,x_2,\ldots x_d)=\sum_{k\in\Z}
  \hat\phi_k^{(1)}(x_2,\ldots,x_d)e^{2\pi i kx_1},
\end{equation} 
where each $\hat\phi_k^{(1)}\in C^\infty(\R^{d-1})$. Here, the
$0^{\mathrm{th}}$ Fourier-function $\hat\phi_0^{(1)}$ has a
particularly nice interpretation: it can be naturally considered as
defined on the nilpotent Lie group $N^{(1)}:=N/N_{(1)}$, or more
precisely, on the the compact nilmanifold $N^{(1)}/\Gamma^{(1)}$ (see
\S\ref{sec:malcev-theory} for these notations).

On the other hand, observe that the basis
$\{v_2+\n_{(1)},v_3+\n_{(1)}, \ldots, v_d+\n_{(1)}\}$ is a Mal'cev one
for $\n/\n_{(1)}$ strongly based on the lattice
$\Gamma^{(i)}=\Gamma/\Gamma_{(i)}$.

That means we can repeat our previous argument to prove that
$\hat\phi_0^{(1)}$ is $\Z$-periodic on its first variable, and hence,
we can consider the Fourier-like development
\begin{displaymath}
  \hat\phi_0^{(1)}(x_2,\ldots,x_d)=\sum_{k\in\Z}\hat\phi_k^{(2)}
  (x_3,\ldots,x_d)e^{2\pi i k x_2}. 
\end{displaymath}

Once again, the Fourier-coefficient function $\hat\phi_0^{(2)}$ can be
considered as an element of $C^\infty(N^{(2)}/\Gamma^{(2)})$ and the
set $\{v_3+\n_{(2)},\ldots,v_d+\n_{(d)}\}$ is a Mal'cev basis strongly
based on $\Gamma^{(2)}$.

By induction, we get a family of \emph{Fourier-like coefficients}
\begin{displaymath}
  \hat\phi_k^{(j)}\in C^\infty(\R^{d-j}), \quad\forall
  j\in\{1,\ldots,d\},\ \forall k\in\Z,
\end{displaymath}
where each $\hat\phi_0^{(j)}\in C^\infty(N^{(i)}/\Gamma^{(i)})\subset
C^\infty(\R^{d-j})$ and satisfies
\begin{displaymath}
  \hat\phi^{(j)}_0(x_{j+1},\ldots
  x_d)=\sum_{k\in\Z}\hat\phi^{(j+1)}_k(x_{j+2},\ldots,x_d)e^{2\pi i k
    x_{j+1}}.
\end{displaymath}

Now we proceed to define the \emph{pseudo-polynomials} on $P$: we
shall say that $\phi\in C^\infty(P)$ is a \emph{pseudo-polynomial}
(with respect to $\mc{V}$) of degree less or equal than $n\in\N$ iff
\begin{displaymath}
  \hat\phi_k^{(j)}\equiv 0, \quad\text{for every } j\in\{1,\ldots,d\}\
  \text{and } \abs{k}>n. 
\end{displaymath}
The linear space of pseudo-polynomials on $P$ will be denoted by
$\Pol(P)$ and we shall write $\Pol_n(P)$ for the subspace of
pseudo-polynomials with degree at most $n$.

Notice that $M=\T\times P$ is a compact nilmanifold itself, hence we
can talk about pseudo-polynomials on $M$. In this case, we shall add
the vector $v_0:=\partial_t$ (which generated the Lie algebra of $\R$)
to the basis $\mc V$, and hence any function $\phi\in C^\infty(M)$
will be written in coordinates $(t,x_1,\ldots,x_d)$, being $\phi$
$\Z$-periodic on its first coordinate, too. So, we can also consider
the Fourier-like development
\begin{displaymath}
  \phi(t,x_1,\ldots,x_d)=\sum_{k\in\Z}\hat\phi_k^{(0)}(x_1,\ldots,x_d)
  e^{2\pi ikt},
\end{displaymath}
with each $\hat\phi_k^{(0)}\in C^\infty(P)$. Of course, by analogy
with our definition of pseudo-polynomials on $P$, we define
\begin{displaymath} 
  \Pol_n(M):=\{\phi\in C^\infty(M) : \hat\phi_0^{(0)}\in\Pol_n(P),\
  \hat\phi_k^{(0)}\equiv 0,\ \forall \abs{k}>n\}.
\end{displaymath}

Combining an inductive argument on the dimension of $M$ with classical
Fourier theory one can easily show
\begin{proposition}
  \label{pro:pseudo-poly-density}
  The linear space 
  \begin{displaymath}
    \Pol(M)=\bigcup_{n\geq 0}\Pol_n(M)
  \end{displaymath}
  is dense in $C^\infty(M)$. In particular, this implies that the family
  $(V_n)$ given by
  \begin{displaymath}
    V_n:=\Pol_n(M)\cap C^\infty_\mu(M),\quad \forall n\geq 1
  \end{displaymath}
  is a filtration of $C^\infty_\mu(M)$.
\end{proposition}

Now we will prove Lemma~\ref{lem:mainLem} assuming the filtration
$(V_n)$ is given by Proposition~\ref{pro:pseudo-poly-density}:

\begin{proof}[Proof of Lemma~\ref{lem:mainLem} in the nilpotent case]
  Let us write $d:=\dim N$. We will recursively define, for
  $k\in\{0,1,\ldots,d\}$, two sequences $\gamma_k\in C^\infty(\T,N)$
  and $(\bar q_k)\subset\N$ satisfying the following condition: there
  exists a constant $C_k\in\R$ such that for every $p\in\Z$ coprime
  with $\bar q_k$ and every $\phi\in V_n$,
  \begin{equation}
    \label{eq:Bs-k-hypo-ind}
    \Bs_{H_k T_{p/\bar q_k}H^{-1}_k}^q\phi(t,x)=C_k
    \hat\phi_0^{(k)}(\gamma_k(t)^{-1}xN^{(k)}),\quad\forall
    (t,x)\in\T\times N,
  \end{equation}
  where $H_k=H_{0,\gamma_k}\subset\Hsw\infty(\T\times P)$ and
  considering $\hat\phi_0^{(k)}$ as a complex function on
  $N^{(k)}=N/N_{(k)}$ (see \S\ref{sec:malcev-theory} for notations).

  At this point it is important to notice that, since $\phi\in
  V_n\subset C^\infty_\mu(M)$, then
  \begin{displaymath}
    \int_P\hat\phi_0^{(k)}\dd\nu=0, \quad\text{for every }  0\leq
    k\leq d. 
  \end{displaymath}
  Thus, in particular, the complex number $\phi_0^{(d)}$ is equal to
  zero, and so, condition \eqref{eq:Bs-k-hypo-ind} for $k=d$ means
  that the Birkhoff sum vanishes. By
  Lemma~\ref{lem:cobound-periodic-map}, this is equivalent to $\phi\in
  B(H_dT_{p/\bar q_d}H_d^{-1}, C^\infty)$.

  Now let us start with the case $k=0$. Observe that without lost of
  generality we can assume $n<q_0$. Let us define $\gamma_0\equiv 1_N$
  (so, $H_0=id_M$) and $\bar q_0:=q_0$. Hence, for every $\phi\in
  V_n$, and every $p\in\Z$ coprime with $\bar q_0$, we have
  \begin{equation}
    \label{eq:Bs-k0-first-step-ind}
    \begin{split}
      \Bs_{H_0T_{p/\bar q_0}H_0^{-1}}^{\bar q_0}
      \phi(t,x)&=\sum_{j=0}^{\bar q_0-1}
      \phi\bigg( t+j\frac{p}{\bar q_0},x\bigg) \\
      & = \sum_{j=0}^{\bar q_0-1} \sum_{\abs{\ell}\leq n}
      \hat\phi_{\ell}^{(0)}(x) e^{2\pi i\ell(t+\frac{j}{\bar q_0})} =
      \bar q_0\phi_0^{(0)}(x),
    \end{split}
  \end{equation}
  for every $(t,x)\in\R\times\R^d$. So, condition
  \eqref{eq:Bs-k-hypo-ind} is verified for $k=0$.

  Now, suppose we have already defined $\gamma_{k-1}\in
  C^\infty(\T,N)$ (and then,
  $H_{k-1}=H_{0,\gamma_{k-1}}\in\Hsw\infty(\T\times G/H)$) and $\bar
  q_{k-1}\in\N$, with $1\leq k\leq d$, then let us construct
  $\gamma_k$ and $\bar q_k$.

  To do this, we start considering an auxiliary $C^\infty$-function
  $\rho\colon [0,1]\to[0,1]$ satisfying the following conditions:
  \begin{enumerate}[$(i)$]
  \item $\rho(t)=0$, for every $t$ in a neighborhood of $0$;
  \item $\rho(t)=1$, for every $t$ in a neighborhood of $1$;
  \item $\rho^\prime(t)\geq 0$, for all $t$;
  \item $\rho(1-t)=1-\rho(t)$, for all $t$.
  \end{enumerate}
  
  Then, we use $\rho$ to define a new auxiliary function
  $\eta\colon[0,1]\to\R$ as follows:
  \begin{displaymath}
    \eta(t):=
    \begin{cases}
      \rho\big(2q_0t-\floor{2q_0t}\big)+
      \frac{\floor{2q_0t}}{q_0}, &
      \text{if } t\in\big[0,\frac{1}{2}\big);\\
      \rho\big(2q_0(1-t)-\floor{2q_0(1-t)}\big)+
      \frac{\floor{2q_0(1-t)}}{q_0},& \text{if }
      t\in\big[\frac{1}{2},1\big]
    \end{cases}
  \end{displaymath}
  In this way, $\eta$ turns to be smooth and it vanishes in some
  neighborhoods of $0$ and $1$, so we can consider it as an element of
  $C^\infty(\T,\R)$. Observe that for any $m\in\Z\setminus\{0\}$ with
  $\abs{m}<q_0$ and any $t\in\T$, it holds
  \begin{equation}
    \label{eq:exp-sums-eta}
    \begin{split}
      \sum_{\ell=0}^{2q_0-1}e^{2\pi im\eta(t+\ell/2q_0)} =
      \sum_{\ell=0}^{q_0-1} \Big(e^{2\pi im(\eta(t)+\ell/q_0)} +
      e^{-2\pi im(\eta(t)+\ell/q_0)}\Big)=0.
    \end{split}
  \end{equation} 

  Then we define
  \begin{displaymath}
    \bar q_k:=2\bar q_{k-1} q_0, 
  \end{displaymath}
  and $\gamma_k\in C^\infty(\T,N)$ by
  \begin{displaymath}
    \gamma_k(t):=\gamma_{k-1}(t)\exp(\eta(\bar
    q_{k-1}t)v_k),\quad\forall t\in\T. 
  \end{displaymath}

  Assuming the inductive hypothesis, let us prove condition
  \eqref{eq:Bs-k-hypo-ind} holds for $k$, too. Let $p\in\Z$ be any
  number coprime with $\bar q_k$, $\phi\in V_n$ arbitrary and $(t,x)$
  be any point in $\T\times N$. Then we have:
  \begin{equation}
    \label{eq:Bs-induct-hypo-k+1}
    \begin{split}
      &\Bs_{H_{k}T_{p/\bar q_{k}}H_{k}^{-1}}^{\bar q_{k}} \phi(t,x) =
      \sum_{j=0}^{\bar q_{k}-1}\phi\bigg(t+\frac{j}{\bar q_k},
      \gamma_{k}\Big(t+\frac{j}{\bar
        q_k}\Big)\gamma_{k}(t)^{-1}x\bigg)
      \\
      &= \sum_{\ell=0}^{2q_0-1}\sum_{j=0}^{\bar q_{k-1}-1}
      \phi\bigg(\Big(t+\frac{\ell}{\bar q_k}\Big)+ \frac{j}{\bar
        q_{k-1}}, \gamma_{k}\bigg(\Big(t+\frac{\ell}{\bar
        q_k}\Big)+\frac{j}{\bar q_{k-1}}\bigg)\gamma_{k}(t)^{-1}x\bigg) \\
      &= C_{k-1}\sum_{\ell=0}^{2q_0-1} \hat\phi_0^{(k-1)}\bigg(
      \exp\bigg(\eta\Big(\bar q_{k-1}t+\frac{\ell}{2q_0}
      \Big)v_{k}\bigg)\gamma_{k}(t)^{-1}xN^{(k-1)}\bigg) \\
      &= C_{k-1} \sum_{\ell=0}^{2q_0-1}\hat\phi_0^{(k-1)} \bigg(\tilde
      x_k+\eta\Big(\bar q_{k-1}t+\frac{\ell}{2q_0}\Big),
      \tilde x_{k+1},\ldots,\tilde x_d\bigg) \\
      &=C_{k-1} \sum_{\ell=0}^{2q_0-1} \sum_{\abs{m}\leq n}
      \hat\phi_m^{(k)}(\tilde x_{k+1},\ldots,\tilde x_d)e^{2\pi i
        m(\tilde x_k+\eta(\bar q_{k-1}t+\ell/2q_0))} \\
      &= C_{k-1} \sum_{\abs{m}\leq n} \hat\phi_m^{(k)}(\tilde
      x_{k+1},\ldots,\tilde x_d) e^{2\pi i m\tilde x_k}
      \sum_{\ell=0}^{2q_0-1} e^{2\pi i m\eta(t+\ell/2q_0)}\\
      &= q_0C_{k-1}\hat\phi_0^{(k)}(\tilde x_{k+1},\ldots,\tilde x_d)
      = q_0C_{k-1}\hat\phi_0^{(k)}\Big(\gamma_k(t)^{-1}xN^{(k)}\Big),
    \end{split}
  \end{equation}
  where the sixth equality is consequence of \eqref{eq:exp-sums-eta}
  and where $(\tilde x_k,\tilde x_{k+1},\ldots,\tilde x_d)$ denotes
  the ``coordinates'' of the point $\gamma_{k+1}(t)^{-1}xN^{(k-1)}$ in
  the Lie algebra $\n^{(k-1)}$, \ie they satisfy the following
  equation:
  \begin{displaymath}
    \tilde x_kv_k+ \tilde x_{k+1}v_{k+1}+\ldots+
    \tilde
    x_dv_d+\n_{(k-1)}=\exp_{N^{(k-1)}}^{-1}(\gamma_{k}(t)^{-1}xN^{(k-1)}).
  \end{displaymath}

  In this way, \eqref{eq:Bs-induct-hypo-k+1} shows condition
  \eqref{eq:Bs-k-hypo-ind} holds for $k$, finishing the proof of the
  lemma.
\end{proof}

\section{The compact case}
\label{sec:proof-compact-case}

We start this section with a geometric construction which is
completely independent of the homogeneous structure of the supporting
manifold. So, for the time being, let us suppose $M$ is an arbitrary
smooth connected closed manifold and $\mu$ any Borel probability
measure on $M$.

\subsection{Equidistributed loops}
\label{sec:equid-loops}

Given a finite dimensional subspace $E\subset C^\infty_\mu(M)$, a
smooth loop $\gamma\in C^\infty(\T,M)$ is said to be
\emph{E-equidistributed} when there exists $m\in\N$ such that
\begin{equation}
  \label{eq:E-equid-dist-paths-def}
  \sum_{j=0}^{m-1}\phi\bigg(\gamma\Big(t+\frac{j}{m}\Big)\bigg)=0,
  \quad\forall t\in\T,\ \forall\phi\in E.
\end{equation}
The number $m$ will be called the \emph{period of the loop}.

\begin{theorem}
  \label{thm:E-equid-dist-existence}
  For every finite dimensional subspace $E\subset C^\infty_\mu(M)$,
  there exists a smooth $E$-equidistributed loop $\theta\in
  C^\infty(\T,M)$.
\end{theorem}

\begin{proof}
  Let $N:=\dim E$ and $(\phi_1)_{i=1}^N$ be a basis of $E$, and define
  $\Phi\colon M\to\R^N$ by $\Phi(x):=(\phi_1(x),\ldots,\phi_N(x))$,
  for every $x\in M$. For each $m\in M$, let us write
  \begin{displaymath}
    \Phi^{(m)}(x_1,\ldots,x_m):=\sum_{j=1}^m\Phi(x_j)\in\R^N, \quad\forall
    (x_1,\ldots,x_m)\in M^m.
  \end{displaymath}

  Let us consider the sets $Y^{(m)},Z^{(m)}\subset M^m$ given by
  \begin{align}
    \label{eq:Y-m-def}
    Y^{(m)}&:=\left\{\bar x\in M^m : D\Phi_m(\bar x)\colon T_{\bar
        x}M^m\to\R^N \text{ is
        surjective}\right\}, \\
    \label{eq:Z-m-def}
    Z^{(m)}&:=\left\{\bar x\in M^m : \Phi_m(\bar x)=0\in\R^N\right\},
  \end{align}
  and then define $X^{(m)}:=Y^{(m)}\cap Z^{(m)}$.

  We divide the rest of the proof in several lemmas:
  
  \begin{lemma}
    \label{lem:Y-N-non-empty}
    For every $n\geq N$, the set $Y^{(n)}$ is non-empty.
  \end{lemma}

  \begin{proof}[Proof of Lemma~\ref{lem:Y-N-non-empty}]
    First observe that, given any $n\geq 1$,
    \begin{equation}
      \label{eq:DPhi-n}
      {D\Phi^{(n)}}_{(x_1,\ldots,x_n)}(v_1,\ldots,v_n)=\sum_{j=1}^n
      D\Phi_{x_j}(v_j),
    \end{equation}
    for every $(x_1,\ldots,x_n)\in M^n$ and every $(v_1,\ldots,v_n)\in
    T_{(x_1,\ldots,x_n)}M^n$. This lets us affirm that for any
    $k,n\in\N$ and any \emph{``forget-some-coordinate''} projection
    $\mathrm{pr}\colon M^{n+k}\to M^n$, it holds
    \begin{equation}
      \label{eq:pr-inv-Y-n}
      \mathrm{pr}^{-1}(Y^{(n)})\subset Y^{(n+k)}.
    \end{equation}
   
    That means it is enough to show $Y^{(N)}$ is non-empty. Reasoning
    by contradiction, suppose this is not the case. By
    \eqref{eq:DPhi-n}, this implies the set
    \begin{displaymath}
      \{D\Phi_{x}(v) : x\in M,\ v\in T_{x}M\}
    \end{displaymath}
    is contained in a proper linear sub-space of $\R^N$, and
    therefore, we can find a non-identically zero linear functional
    $\mc{L}\colon\R^N\to\R$ such that $D(\mc{L}\circ\Phi)_x=0$, for
    every $x\in M$. Of course, since we are assuming $M$ is connected,
    this implies $\mc{L}\circ\Phi\colon M\to\R^N$ is a constant
    function. Since the coordinate functions of $\Phi$ (\ie functions
    $\phi_1,\ldots,\phi_N$) have zero integral with respect to $\mu$,
    we conclude that $\mc{L}\circ\Phi\equiv 0$, contradicting the
    linear independence of the set $(\phi_i)_{i=1}^N$. So,
    $Y^{(N)}\neq\emptyset$, and by \eqref{eq:pr-inv-Y-n}, we get
    $Y^{(n)}\neq\emptyset$, for every $n\geq N$.
  \end{proof}

  \begin{lemma}
    \label{lem:X-m-non-empty}
    There exists $m\in\N$ such that $X^{(m)}$ is non-empty.
  \end{lemma}
  \begin{proof}[Proof of Lemma~\ref{lem:X-m-non-empty}]
    Consider the set
    \begin{displaymath}
      C_\Phi :=\bigcup_{n\geq 1}
      \bigg\{\sum_{j=1}^n\lambda_j\Phi(x_j)\in\R^N : x_j\in M,\
      \lambda_j>0,\ \forall j\in\{1,\ldots,n\}\bigg\}.
    \end{displaymath}
    Observe $C_\Phi$ is a convex cone in $\R^N$. We claim
    $C_\Phi=\R^N$. In fact, if this would not be the case, then there
    should exist a non-null linear functional $\mc{L}\colon\R^N\to\R$
    such that $\mc{L}(y)\geq 0$ for every $y\in C_\Phi\subset\R^N$
    and, in particular, $\mc{L}(\Phi(x))\geq 0$, for every $x\in M$.
    But since the coordinate functions of $\Phi$ belong to
    $C^\infty_\mu(M)$, it holds
    \begin{displaymath}
      \int_M\mc{L}(\Phi(x))\dd\mu=0.
    \end{displaymath}
    Hence, $\mc{L}\circ\Phi$ should be identically equal to zero,
    contradicting the linear independence of the coordinate functions
    $(\phi_i)_{i=1}^N$. Thus, $C_\Phi=\R^N$.

    By Lemma~\ref{lem:Y-N-non-empty}, $Y^{(N)}$ is non-empty, so we
    can consider an arbitrary point $(z_1,\ldots,z_N)\in Y^{(N)}$. By
    our previous assertion about $C_\Phi$, there exist $n\in\N$,
    positive numbers $\lambda_1,\ldots,\lambda_n$ and points
    $x_1,\ldots,x_n\in M$ such that
    \begin{equation}
      \label{eq:lambda-j}
      \sum_{j=1}^n\lambda_j\Phi(x_j) = -\Phi_N(z_1,\ldots,z_N)=
      -\sum_{j=1}^N\Phi(z_j). 
    \end{equation}

    Now, since $\Phi_N(Y^{(N)})$ is open in $\R^N$, we can assume (up
    to an arbitrary small perturbation of the points $z_1,\ldots,z_N$)
    that each $\lambda_j\in\Q$, and hence we can find
    $p_1,\ldots,p_n,q\in\N$ such that $\lambda_j=p_j/q$, for each
    $1\leq j\leq n$. Now, we define $m:=qN+\sum_{1\leq j\leq n} p_j$
    and we claim $X^{(m)}\neq\emptyset$. In fact, if we define
    $(w_1,\ldots,w_m)\in M^m$ by
    \begin{displaymath}
      w_j:=
      \begin{cases}
        z_{\ceil{j/q}}, & \text{if } 1\leq j\leq qN,\\
        x_k, & \text{if } qN< j\leq qN+\sum_{\ell=1}^kp_\ell,
      \end{cases}
    \end{displaymath}
    from \eqref{eq:lambda-j} we easily conclude $(w_1,\ldots,w_m)\in
    X^{(m)}$. 
  \end{proof}

  Now, for each $n\geq 2$, let us consider the diffeomorphism
  $\sigma_n\colon M^n\to M^n$ given by
  \begin{displaymath}
    \sigma_n(x_1,x_2,\ldots,x_n):=(x_2,\ldots,x_n,x_1),\quad \forall
    (x_1,\ldots,x_n)\in M^n.
  \end{displaymath}

  And we shall prove our last
  \begin{lemma}
    \label{lem:cc-Xm}
    There exist $m\geq 1$ and $\bar z\in X^{(m)}$ such that
    $\sigma_m(\bar z)\in\cc(X^{(m)},\bar z)$.
  \end{lemma}

  \begin{proof}[Proof of Lemma~\ref{lem:cc-Xm}]
    Let $m_0$ be a natural number such that $X^{(m_0)}$ is non-empty,
    and let $\bar z=(x_1,\ldots,x_{m_0})$ be any point in
    $X^{(m_0)}$. For each $q\in\N$, let us consider the point $\bar
    z^{(q)}=(z^{(q)}_1,z^{(q)}_2,\ldots,z^{(q)}_{qm_0})\in X^{(qm_0)}$
    given by
    \begin{displaymath}
      z^{(q)}_j:= x_{\ceil{j/q}}, \quad\text{for } 1\leq j\leq qm_0.
    \end{displaymath}
    
    We claim that $\sigma_{qm_0}(\bar z^{(q)})\in\cc(X^{(qm_0)},\bar
    z^{(q)})$, provided $q$ is sufficiently large. To prove this, we
    shall construct a continuous curve $\rho\colon [0,1]\to
    X^{(qm_0)}\subset M^{qm_0}$, with $\rho(0)=\bar z^{(q)}$ and
    $\rho(1)=\sigma_{qm_0}(\bar z^{(qm_0)})$.

    For each $1\leq j\leq qm_0$, we write $\rho_j\colon[0,1]\to M$
    for the $j^{\mathrm{th}}$-coordinate function of $\rho$, \ie
    $\rho(t)=(\rho_1(t),\rho_2(t), \ldots,\rho_{qm_0}(t))\in
    M^{qm_0}$.

    We start defining $\rho$ on the interval $[0,1/2]$. To do that,
    fitst let us consider continuous paths
    $\alpha=(\alpha_1,\ldots,\alpha_{m_0})\colon[0,1]\to M^{m_0}$ such
    that
    \begin{align*}
      \alpha_i(0)&=x_i,&\quad\alpha_i(1)&=x_{i+1},&\quad\text{for } 1\leq
      i<m_0, \\
      \alpha_{m_0}(0)&=x_{m_0},&\quad\alpha_{m_0}(1)&=x_1.&
    \end{align*}

    Now, we choose a (small) neighborhood $U$ of $\bar z$ in $M^{m_0}$
    such that $U\subset Y^{(m_0)}$ and $U\cap X^{(m_0)}$ is
    connected. Since $\Phi^{(m_0)}$ is a submersion on $Y^{(m_0)}$, we
    can find a continuous path $\beta\colon[0,1]\to U\subset
    M^{m_0}$ satisfying $\beta(0)=\bar z$ and
    \begin{equation}
      \label{eq:beta-path-def}
      \Phi^{(m_0)}\big(\beta(t)\big)=-\frac{\Phi^{(m_0)}\big(\alpha(t))}{q},
      \quad\forall t\in [0,1],
    \end{equation}
    provided $q$ is sufficiently large. Notice that, since
    $\alpha(1)\in X^{(m_0)}$, then $\beta(1)$ also belongs to
    $X^{(m_0)}$ itself.

    Then, we define each coordinate function of the path $\rho$ on
    $[0,1/2]$ by
    \begin{displaymath}
      \rho_j(t):=
      \begin{cases}
        \beta_{\ceil{j/q}}(2t),& \text{if } 1\leq j< m_0q,\text{ and
        }j\not\in q\Z \\
        \alpha_{j/q}(2t),& \text{if } j\in\{q,2q,\ldots,m_0q\},
      \end{cases}
    \end{displaymath}
    and every $t\in[0,1/2]$.  Notice that, as a consequence of
    \eqref{eq:beta-path-def}, $\rho(t)\in X^{(qm_0)}$, for every
    $t\in[0,1/2]$.

    In order to define path $\rho$ on $[1/2,1]$, let us consider a
    continuous path $\gamma\colon[0,1]\to X^{(m_0)}$ joining
    $\beta(1)$ to $\bar z$. Such a path $\gamma$ does exist because
    both points $\beta(1)$ and $\bar z$ belong to $U\cap X^{(m_0)}$,
    which is a connected open set of the smooth manifold $X^{(m_0)}$,
    and hence, it is arc-wise connected.

    Finally, we define $\rho$ on $[1/2,1]$ by
    \begin{displaymath}
      \rho_j(t):=
      \begin{cases}
        \gamma_{\ceil{j/q}}(2t-1),& \text{if } j\not\in q\Z\\
        x_{j+1},& \text{if } j\in q\Z\text{ and } 1\leq j<qm_0, \\
        x_1,& j=qm_0.
      \end{cases}
    \end{displaymath}
    In this way, $\rho$ is clearly a continuous path contained in
    $X^{(qm_0)}$ and joins $\rho(0)=\bar z$ to
    $\rho(1)=\sigma_{qm_0}(\bar z)$, as desired.
  \end{proof}

  Finally, let $m$ and $\bar z\in X^{(m)}$ as in
  Lemma~\ref{lem:cc-Xm}. Since $X^{(m)}\subset M^m$ is a
  $\sigma_m$-invariant embedded submanifold, and $\sigma_m$ is an
  $m$-periodic diffeomorphism, we can find a smooth loop
  $\tilde\theta\in C^\infty(\T,X^{(m)})$ satisfying
  \begin{equation}
    \label{eq:theta-0-propert}
    \tilde\theta\bigg(t+\frac{1}{m}\bigg)=\sigma_m(\tilde\theta(t)),\quad 
    \forall t\in\T.
  \end{equation}
  Then, for any $t\in\T$, if we write
  $\theta(t)=(\theta_1(t),\ldots,\theta_m(t))\in M^m$, it holds
  \begin{displaymath}
    \R^N\ni 0=\Phi^{(m)}(\theta_1(t),\ldots,\theta_m(t))=
    \sum_{j=1}^m\Phi(\theta_j(t)) =
    \sum_{j=1}^m\Phi\bigg(\theta_1\Big(t+\frac{j}{m}\Big)\bigg), 
  \end{displaymath}
  where last equality is consequence of
  \eqref{eq:theta-0-propert}. 

  Thus, $\theta=\theta_1$ is a smooth $E$-equidistributed loop, as
  desired.
\end{proof}

\subsection{The filtration in the compact case}
\label{sec:choosing-filt-compact}

In this section we construct the filtration of $C^\infty_\mu(\T\times
P)$ in order to prove Lemma~\ref{lem:mainLem}.

To do this, we return to our homogeneous setting, assuming $G$ is a
compact (connected) Lie group, $H<G$ a closed subgroup, $P=G/H$ and
$M=\T\times P$. For the sake of simplicity of the exposition, we start
assuming $H=\{1_G\}$. The general case will be easily gotten from this
particular one.

If $\nu_G$ denotes the Haar (probability) measure on $G$, there are
two unitary representations of $G$ on $L^2_0(G,\nu_G):=\{\phi\in
L^2(G,\nu_G) : \int\phi\nu_G=0\}$ given by
\begin{equation}
  \begin{split}
    (L_g\phi)(x)&:=\phi(g^{-1}x),\\
    (R_g\phi)(x)&:=\phi(xg), \quad\forall g,x\in G,\ \forall\phi\in
    L^2_0(G,\nu_G).
  \end{split} 
\end{equation}

By the classical Peter-Weyl theorem, we know left action $L$
decomposes in a direct sum of finite-dimensional irreducible
sub-representations, \ie there exists a family $(E_n)_{n\geq 1}$ of
finite-dimension subspaces of $L^2_0(G,\nu_G)$ such that
$\bigoplus_{n} E_n$ is dense in $L^2_0(G,\nu_G)$ and each $E_n$ is
$L$-invariant, with no proper $L$-invariant subspace contained in
$E_n$. Moreover, these spaces satisfy $E_n\subset
C^\infty_{\nu_G}(G)=C^\infty(G)\cap L^2_0(G,\nu_G)$, for every $n\geq
1$ and they are also $R$-invariant (for instance, see \S 3.3 in
\cite{SepanskiCompLieGr} for details).

In particular, this implies that, if $\gamma\colon\T\to G$ an
$E_n$-equidistributed with period $m$, then
\begin{equation}
  \label{eq:En-equidist-loops}
  \sum_{j=0}^{m-1}\phi\bigg(\gamma\Big(t+\frac{j}{m}\Big)x\bigg) =
  \sum_{j=0}^{m-1}(R_x
  \phi)\bigg(\gamma\Big(t+\frac{j}{m}\Big)\bigg)=0, 
\end{equation}
for every $t\in\T$, every $x\in G$ and any $\phi\in E_n$.

Now, for each $\phi\in C^\infty(M)$ and every $k\in\Z$, we define
$\hat\phi_k\in C^\infty(G)$ by
\begin{equation}
  \label{eq:hat-phi-k-def-compact}
  \hat\phi_k(x):=\int_\T\phi(t,x)e^{-2\pi ikt}\dd t,\quad\forall x\in
  G,
\end{equation}
and
\begin{equation}
  \label{eq:filtration-def-compact-case}
  V_n:=\left\{\phi\in C^\infty_\mu(M) : \hat\phi_0\in\bigoplus_{j\leq
      n} E_j,\ \hat\phi_k\equiv 0,\ \forall \abs{k}>n\right\}.  
\end{equation}

By Peter-Weyl theorem and classical Fourier series arguments we have
\begin{lemma}
  \label{lem:V_l-filtration-compact-case}
  The family $(V_n)$ given by \eqref{eq:filtration-def-compact-case}
  is a filtration for $C^\infty_\mu(M)$.
\end{lemma}

\begin{proof}[Proof of Lemma~\ref{lem:mainLem} in the compact case]
  Let us consider the filtration $(V_j)_{j\geq 1}$ given by
  \eqref{eq:filtration-def-compact-case}. As we did in the nilpotent
  case, without lost of generality we can assume $n<q_0$.

  Let $\tilde\gamma\in C^\infty(\T,G)$ be a $\left(\bigoplus_{j\leq
      k}E_j\right)$-equidistributed loop in $G$, and let $\tilde m$ be
  its period. Then let us define $\gamma\colon\T\to G$ by
  $\gamma(t):=\tilde\gamma(q_0t)$ and write $\bar q :=q_0\tilde
  m$. Notice $H_{0,\gamma}\in\Hsw\infty(\T\times G)$ clearly commutes
  with $T_{1/q_0}$. Let us show that condition (ii) of
  Lemma~\ref{lem:mainLem} also holds.

  To do that, let $p$ be any integer coprime with $\bar q$ and let us
  consider an arbitrary $\phi\in V_n$. Once again we consider the
  Fourier-like development of $\phi$:
  \begin{displaymath}
    \phi(t,x)=\sum_{\abs{\ell}\leq n}\hat\phi_\ell(x)e^{2\pi i\ell t},
    \quad\forall (t,x)\in\T\times G,
  \end{displaymath}
  where each $\hat\phi_\ell\in C^\infty(G)$ is given by
  \eqref{eq:hat-phi-k-def-compact} and $\hat\phi_0\in\bigoplus_{j\leq
    n}E_j$.

  Then we have,
  \begin{equation}
    \label{eq:Vn-cobound-compact-case}
    \begin{split}
      \Bs_{H_{0,\gamma}T_{p/\bar q}H_{0,\gamma}^{-1}}^{\bar
        q}&\phi(t,x) =\sum_{j=0}^{\bar q-1} \phi\bigg(t+\frac{j}{\bar
        q},
      \gamma\Big(t+\frac{j}{\bar q}\Big)x\bigg)\\
      &=\sum_{j=0}^{\bar q-1}\phi\bigg(t+\frac{j}{q}, \tilde\gamma
      \Big(q_0t+\frac{j}{\tilde m}\Big)x\bigg) \\
      &= \sum_{j=0}^{\bar q-1}\sum_{\abs{\ell}\leq
        n}\hat\phi_\ell\bigg(\tilde\gamma \Big(q_0t+\frac{j}{\tilde
        m}\Big)x\bigg)e^{2\pi i\ell (t+\frac{j}{\bar q})} \\
      &= \sum_{j=0}^{q_0-1}\sum_{k=0}^{\tilde m-1}
      \sum_{\abs{\ell}\leq n} \hat\phi_\ell\bigg(\tilde\gamma
      \Big(q_0t+\frac{k}{\tilde m}\Big)x\bigg) e^{2\pi i\ell
        (t+\frac{j}{q_0}+\frac{k}{\tilde m})}\\
      &=\sum_{\abs{\ell}\leq n}e^{2\pi i\ell t}\sum_{k=0}^{\tilde m-1}
      \hat\phi_\ell\bigg(\tilde\gamma \Big(q_0t+\frac{k}{\tilde m}
      \Big)x\bigg)e^{2\pi i\ell\frac{k}{\tilde m}}
      \sum_{j=0}^{q_0-1}e^{2\pi i\ell\frac{j}{q_0}} \\
      & = q_0\sum_{k=0}^{\tilde m-1}\hat\phi_0\bigg(\tilde\gamma
      \Big(q_0t+\frac{k}{\tilde m}\Big)x\bigg) \\
      &= q_0\sum_{k=0}^{\tilde m-1}R_x\hat\phi_0\bigg(\tilde\gamma
      \Big(q_0t+\frac{k}{\tilde m}\Big)\bigg) =0
    \end{split}
  \end{equation} 
  for every $t\in\T$ and every $x\in G$, and where the last equality
  is a consequence \eqref{eq:En-equidist-loops} and invariance by $R_x$ of $\bigoplus_{j\leq
    n}E_j$.

  Thus, by Lemma~\ref{lem:cobound-periodic-map},  it follows from
  \eqref{eq:Vn-cobound-compact-case} that $V_n\subset
  B(H_{0,\gamma}T_{p/q}H_{0,\gamma}^{-1})$, as desired.
\end{proof}

\subsection{The case $H\neq\{1_G\}$ }
\label{sec:H-non-trivial}

Now, let us suppose $H<G$ is a proper closed subgroup. Since
$G$ and $H$ are both compact, they admit unique Haar probability
measures, which will be denoted by $\nu_G$ and $\nu_H$,
respectively. The Haar measure on $G/H$ will be simply denoted by
$\nu$.

We will write $\pi_H\colon G\to G/H$ for the canonical projection and
we can define the linear operator $\Pi_H\colon C^\infty(G)\to
C^\infty(G/H)$ by 
\begin{displaymath}
  \Pi_H\psi(gH):=\int_H\psi(gx)\dd\nu_H(x),\quad\forall \psi\in
  C^\infty(G),
\end{displaymath}

Let us remark that $\Pi_H$ is continuous, closed and surjective (in
fact, the pull-back by $\pi_H$ is a section of $\Pi_H$) and satisfies
$\Pi_H(C^\infty_{\nu_G}(G))=C^\infty_\nu(G/H)$. In particular, the
family $(\Pi_H(E_j))_{j\geq 1}$, where spaces $E_j$ are defined as in
\S\ref{sec:choosing-filt-compact}, turns to be a filtration of
$C^\infty_\nu(G/H)$, where each $\Pi_H(E_j)$ has finite dimension.

Then, we have the following 
\begin{lemma}
  \label{lem:E-equid-on-G/H}
  If $\gamma\in C^\infty(\T,G)$ is an $E_k$-equidistributed loop (with
  $k\in\N$ arbitrary), then $\pi_H\circ\gamma$ is a
  $\Pi_H(E_k)$-equidistributed loop on $G/H$.
\end{lemma}

\begin{proof}
  Let $m$ denote the period of $\gamma$ and $\phi\in E_k$ be
  arbitrary. Then we have
  \begin{displaymath}
    \begin{split}
      \sum_{j=0}^{m-1}\Pi_H(\phi)&\bigg(\pi_H\circ\gamma
      \Big(t+\frac{j}{m}\Big)\bigg) = \sum_{j=0}^{m-1}\int_H
      \phi\bigg(\gamma \Big(t+\frac{j}{m}\Big)y\bigg)\dd\nu_H(y) \\
      &=\int_H \Bigg(\sum_{j=0}^{m-1} \phi\bigg(\gamma
      \Big(t+\frac{j}{m}\Big)y\bigg)\Bigg)\dd\nu_H(y) \\
      &=\int_H \Bigg(\sum_{j=0}^{m-1} (R_y\phi)
      \bigg(\gamma\Big(t+\frac{j}{m}\Big)\bigg)\Bigg)\dd\nu_H(y) = 0,
    \end{split}
  \end{displaymath}
  where last equality is consequence of the $R$-invariance of space
  $E_k$. 
\end{proof}

Now, using Lemma~\ref{lem:E-equid-on-G/H} we can easily extend our
proof of \S\ref{sec:choosing-filt-compact} to the case where $H$ is a
proper subgroup. In fact, given any $\phi\in C^\infty(\T\times G/H)$
and any $k\in\Z$, once again we can define $\hat\phi_k\in
C^\infty(G/H)$ by
\begin{displaymath}
  \hat\phi_k(gH):=\int_\T\phi(t,gH)e^{2\pi ikt}\dd t, \quad\forall
  gH\in G/H,
\end{displaymath}
and so (re)define the filtration $(V_n)_{n\in\N}$ of
$C^\infty_\mu(\T\times G/H)$ analogously to
\eqref{eq:filtration-def-compact-case}:
\begin{displaymath}
  V_n:=\left\{\phi\in C^\infty_\mu(\T\times G/H) :
    \hat\phi_0\in\bigoplus_{j\leq n} \Pi_H(E_j),\ \hat\phi_k\equiv 0,\
    \forall \abs{k}>n\right\},\quad\forall n\in\N, 
\end{displaymath}

Then, invoking Lemma~\ref{lem:E-equid-on-G/H} and the above
filtration, \emph{mutatis mutandis} we can extend the proof of
Lemma~\ref{lem:mainLem} we did in \S\ref{sec:choosing-filt-compact}
 in the case $H=\{1_G\}$ to the general one.

\bibliographystyle{amsalpha} \bibliography{base-biblio}

\providecommand{\bysame}{\leavevmode\hbox to3em{\hrulefill}\thinspace}
\providecommand{\MR}{\relax\ifhmode\unskip\space\fi MR }
\providecommand{\MRhref}[2]{%
  \href{http://www.ams.org/mathscinet-getitem?mr=#1}{#2}
}
\providecommand{\href}[2]{#2}
\begin{thebibliography}{MOP77}

\bibitem[AK11]{AviKocCohomoEqInvDistCirc}
Artur Avila and Alejandro Kocsard, \emph{Cohomological equations and invariant
  distributions for minimal circle diffeomorphisms}, Duke Math. J. \textbf{158}
  (2011), no.~3, 501--536. \MR{2805066}

\bibitem[CG90]{CorwinGreenleaf}
Lawrence~J. Corwin and Frederick~P. Greenleaf, \emph{Representations of
  nilpotent {Lie} groups and their applications. {Part I}}, Cambridge Studies
  in Advanced Mathematics, vol.~18, Cambridge University Press, Cambridge,
  1990, Basic theory and examples. \MR{1070979 (92b:22007)}

\bibitem[FF03]{FlamForniInvDistrHorocycle}
Livio Flaminio and Giovanni Forni, \emph{Invariant distributions and time
  averages for horocycle flows}, Duke Math. J. \textbf{119} (2003), no.~3,
  465--526. \MR{2003124 (2004g:37039)}

\bibitem[FF07]{FlamForniNilflows}
\bysame, \emph{On the cohomological equation for nilflows}, Journal of Modern
  Dynamics \textbf{1} (2007), no.~1, 37--60. \MR{2261071 (2008h:37003)}

\bibitem[For08]{ForniKatokConj}
Giovanni Forni, \emph{On the {Greenfield-Wallach} and {Katok} conjectures in
  dimension three}, Contemporary Mathematics \textbf{469} (2008), 197--213.

\bibitem[Kat01]{KatokRobinson}
Anatole Katok, \emph{Cocycles, cohomology and combinatorial constructions in
  ergodic theory}, Smooth ergodic theory and its applications ({Seattle}, {WA},
  1999), Proc. Sympos. Pure Math., vol.~69, American Mathematical Society,
  Providence, RI, 2001, In collaboration with E. A. Robinson, Jr.,
  pp.~107--173. \MR{1858535 (2003a:37010)}

\bibitem[Mal49]{MalcevHomogenSpaces}
A.~I. Mal{'}cev, \emph{On a class of homogeneous spaces}, Izvestiya Akad. Nauk.
  SSSR. Ser. Mat. \textbf{13} (1949), 9--32. \MR{0028842 (10,507d)}

\bibitem[MOP77]{MoulinPinchonSyst}
Jean Moulin-Ollagnier and Didier Pinchon, \emph{Syst\`emes dynamiques
  topologiques. {I}. \'{E}tude des limites de cobords}, Bull. Soc. Math. France
  \textbf{105} (1977), no.~4, 405--414.

\bibitem[NT12]{NavasTriestinoInvDist}
Andr\'es Navas and Michele Triestino, \emph{On the invariant distributions of
  {$C^2$} circle diffeomorphisms of irrational rotation number}, to appear in
  Mathematische Zeitschrift, 2012.

\bibitem[RH12]{fede}
Federico Rodr\'\i{}guez-Hertz, personal communication, 2012.

\bibitem[Sep07]{SepanskiCompLieGr}
Mark~R. Sepanski, \emph{Compact {Lie} groups}, Graduate Texts in Mathematics,
  vol. 235, Springer, New York, 2007.

\end{thebibliography}

\end{document}